\documentclass{amsart}
\usepackage{amsmath} 
\usepackage{amsfonts} 
\usepackage{amssymb} 
\usepackage[utf8]{inputenc} 
\usepackage[british]{babel} 
\usepackage[protrusion=true,expansion=false]{microtype} 
\usepackage{mathtools} 
\usepackage[T1]{fontenc} 
\usepackage{amsthm} 
\usepackage{dsfont} 
\usepackage{mathrsfs} 
\usepackage{xcolor} 
\usepackage{lmodern} 
\usepackage{hyperref} 
\usepackage{enumitem} 
\usepackage[capitalise,noabbrev]{cleveref} 

\definecolor{dblue}{RGB}{0,0,178}
\definecolor{dgreen}{RGB}{31,183,41}
\hypersetup{
	unicode=true,
	colorlinks=true,
	citecolor=dgreen,
	linkcolor=dblue,
	anchorcolor=dblue,
	urlcolor=dblue
}

\newcommand*{\GCH}{\mathsf{GCH}}
\newcommand*{\ZFC}{\mathsf{ZFC}}

\newcommand*{\power}{\mathscr{P}}

\newcommand*{\calA}{\mathcal{A}}
\newcommand*{\calF}{\mathcal{F}}
\newcommand*{\calI}{\mathcal{I}}
\newcommand*{\calU}{\mathcal{U}}
\newcommand*{\calV}{\mathcal{V}}

\newcommand*{\bbP}{\mathbb{P}}
\newcommand*{\bbQ}{\mathbb{Q}}
\newcommand*{\bbR}{\mathbb{R}}

\newcommand*{\ddbbP}{\dot{\mathbb{P}}}
\newcommand*{\ddbbQ}{\dot{\mathbb{Q}}}
\newcommand*{\ddbbR}{\dot{\mathbb{R}}}

\newcommand*{\1}{\mathds{1}}

\DeclareMathOperator{\Add}{Add}
\DeclareMathOperator*{\bigast}{\resizebox{0.9em}{!}{$\ast$}}
\DeclareMathOperator{\cf}{cf}
\DeclareMathOperator{\crit}{crit}
\DeclareMathOperator{\dom}{dom}
\DeclareMathOperator{\Fn}{Fn}
\DeclareMathOperator{\id}{id}
\DeclareMathOperator{\Ord}{Ord}
\DeclareMathOperator{\Ult}{Ult}

\newcommand*{\ddA}{\dot{A}}
\newcommand*{\ddF}{\dot{F}}
\newcommand*{\ddx}{\dot{x}}

\newcommand*{\abs}[1]{\lvert#1\rvert}

\newcommand*{\SetSymbol}[1][]{%
	\;#1\vert\;
	\allowbreak
	\mathopen{}
}

\DeclarePairedDelimiterX{\Set}[1]{\{}{\}}{%
	\renewcommand\mid{\SetSymbol[\delimsize]}
	#1
}

\newcommand*{\tup}[1]{\langle#1\rangle}

\newcommand*{\comp}{\mathrel{\parallel}}
\newcommand*{\defeq}{\mathrel{\vcenter{\baselineskip0.5ex \lineskiplimit0pt
                     \hbox{\scriptsize.}\hbox{\scriptsize.}}}%
                     =}
\newcommand*{\forces}{\mathrel{\Vdash}}
\newcommand*{\iter}{\mathbin{\ast}}
\newcommand*{\mitchless}{\mathrel{\triangleleft}}
\newcommand{\premin}[1]{{}^-{#1}}
\newcommand*{\res}{\nobreak\mskip2mu\mathpunct{}\nonscript
  \mkern-\thinmuskip{\upharpoonright}\mskip6muplus1mu\relax} 
\newcommand*{\stfd}[1]{\Set{\tup{f,\delta,#1}}}
\newcommand*{\vphi}{\varphi}

\newcommand*{\hi}{\hat{\imath}}
\newcommand*{\hj}{\hat{\jmath}}
\newcommand*{\tj}{\tilde{\jmath}}

\newtheoremstyle{boldrk}
  {}{}
  {}{}
  {\bfseries}{.}
  {5pt plus 1pt minus 1pt}{}

\theoremstyle{plain}

\newtheorem{thm}{Theorem}[section]
\newtheorem{claim}{Claim}[thm]
\newtheorem{cor}[thm]{Corollary}
\newtheorem{lem}[thm]{Lemma}
\newtheorem{prop}[thm]{Proposition}

\newtheorem*{factstar}{Fact}

\newtheorem{thmalph}{Theorem} 

\newtheorem{thmalphintro}{Theorem} 

\theoremstyle{boldrk}

\newtheorem{qn}[thm]{Question}

\newtheorem*{rk}{Remark} 

\theoremstyle{definition}

\newtheorem{defn}[thm]{Definition}

\newenvironment*{poc}[1][Proof of Claim]{\begin{proof}[#1]}{\end{proof}}

\title[Proper classes of maximal \(\theta\)-independent families]{Proper classes of maximal \(\theta\)-independent\\families from large cardinals}
\author{Calliope Ryan-Smith}

\email{c.Ryan-Smith@leeds.ac.uk}
\urladdr{https://academic.calliope.mx}
\address{School of Mathematics, University of Leeds, LS2 9JT, UK}

\date{\today{}}
\thanks{The author's work was financially supported by EPSRC via the Mathematical Sciences Doctoral Training Partnership [EP/W523860/1]. For the purpose of open access, the author has applied a Creative Commons Attribution (CC BY) licence to any Author Accepted Manuscript version arising from this submission. No data are associated with this article.}

\begin{document}

\begin{abstract}
While maximal independent families can be constructed in \(\ZFC\) via Zorn's lemma, the presence of a maximal \(\sigma\)-independent family already gives an inner model with a measurable cardinal, and Kunen has shown that from a measurable cardinal one can construct a forcing extension in which there is a maximal \(\sigma\)-independent family. We extend this technique to construct proper classes of maximal \(\theta\)-independent families for various uncountable \(\theta\). In the first instance, a single \(\theta^+\)-strongly compact cardinal has a set-generic extension with a proper class of maximal \(\theta\)-independent families. In the second, we take a class-generic extension of a model with a proper class of measurable cardinals to obtain a proper class of \(\theta\) for which there is a maximal \(\theta\)-independent family.
\end{abstract}

\maketitle

\section{Introduction}
In 1983 Kunen published \cite{kunen_maximal_1983}, a paper exhibiting the equiconsistency of a single measurable cardinal and a single maximal \(\sigma\)-independent family. For an infinite cardinal \(\theta\) and an infinite set \(X\), \(\calA\subseteq\power(X)\) is \emph{\(\theta\)-independent} if \(\abs{\calA}\geq\theta\) and for all partial functions \(p\colon\calA\to2\) with \(\abs{p}<\theta\),
\begin{equation*}
\bigcap\Set*{A\mid p(A)=1}\cap\bigcap\Set*{X\setminus A\mid p(A)=0}\neq\emptyset,
\end{equation*}
where we say \(\sigma\)-independent to mean \(\aleph_1\)-independent. By Zorn's Lemma one obtains maximal \(\omega\)-independent families from \(\ZFC\) alone, but the existence of even a single \(\sigma\)-independent family entails an inner model with a measurable cardinal, a fascinating increase in consistency strength.

The proof that converts a measurable cardinal into a maximal \(\sigma\)-independent family can be extended to larger cardinal properties, something that was known at the time: In \cite{kunen_maximal_1983} Kunen comments that a single strongly compact cardinal \(\kappa\) would beget, in a forcing extension, maximal \(\sigma\)-independent families \(\calA\subseteq\power(\lambda)\) for all \(\lambda\) such that \(\cf(\lambda)\geq\kappa\). We shall prove this result whilst reducing the consistency strength requirement to \(\kappa\) merely being \(\aleph_1\)-strongly compact, and generalise the setting to \(\theta\)-independence.

\begin{thmalphintro}[{\cref{s:applications;ss:strongly-compact}}]
Let \(\kappa\) be a strong limit and \(\theta^+\)-strongly compact for some regular uncountable \(\theta<\kappa\), and let \(G\) be \(V\)-generic for \(\Add(\theta,\kappa)\). In \(V[G]\), for all \(\lambda\geq\kappa\) with \(\cf(\lambda)\geq\kappa\), there is a maximal \(\theta\)-independent family \(\calA\subseteq\power(\lambda)\).
\end{thmalphintro}

We also extend the technique to the case that there is a proper class of measurable cardinals, iterating the process to induce a model in which, for all ground-model measurable cardinals \(\kappa\), there is a maximal \(\kappa\)-independent family \(\calA\subseteq\power(\kappa)\) in the forcing extension. An analysis of the iteration also shows that the Mitchell rank of cardinals is very nearly preserved.

\begin{thmalphintro}[{\cref{s:applications;ss:class-of-measurables}}]
Let \(V\) be a model of \(\ZFC+\GCH\). There is a class-length forcing iteration \(\bbP\) preserving \(\ZFC+\GCH\) such that, if \(G\subseteq\bbP\) is \(V\)-generic, then whenever \(\kappa\) is a measurable cardinal in \(V\) there is a maximal \(\kappa\)-independent family \(\calA\subseteq\power(\kappa)\) in \(V[G]\). Furthermore, whenever \(\kappa\) is a measurable cardinal in \(V\), \(o(\kappa)^V=1+o(\kappa)^{V[G]}\), and whenever \(\kappa\) is non-measurable in \(V\) it remains non-measurable in \(V[G]\).
\end{thmalphintro}

\subsection{Structure of the paper}

In \cref{s:preliminaries} we go over preliminaries, partially to set our notational conventions. An existing familiarity with forcing and ultrapower embeddings will be helpful for understanding the content of the proofs. In \cref{s:tools} we present two powerful tools, our proverbial hammers, for constructing models with maximal \(\theta\)-independent families. In \cref{s:applications} we apply our hammers to various nails: Kunen's theorem (\cref{thm:kunens-theorem-applicable}) and our extensions (\cref{thm:main-sc} and \cref{thm:measurable-class-families}). We end in \cref{s:future} with some open questions.

\section{Preliminaries}\label{s:preliminaries}

We work in \(\ZFC\). For cardinals \(\kappa\leq\lambda\), let \({\power_\kappa(\lambda)=\Set{X\subseteq\lambda\mid\abs{X}<\kappa}}\). If \(V\) is a model of \(\ZFC\) then by \((\kappa^\lambda)^V\), \(\power(X)^V\), etc.\ we mean respectively the set of functions \(\lambda\to\kappa\) in \(V\), the set of subsets of \(X\) in \(V\), etc. \(\id\) is the identity function on an appropriate domain.

\subsection{Independent families}

For regular \(\theta\) and \(\abs{X}\geq\theta\), a \emph{\(\theta\)-independent family on \(X\)} is \(\calA\subseteq\power(X)\) such that \(\abs{\calA}\geq\theta\) and, for all partial functions \(p\colon\calA\to2\) with \(\abs{p}<\theta\),
\begin{equation*}
\calA^p\defeq\bigcap\Set*{A\mid p(A)=1}\cap\bigcap\Set*{X\setminus A\mid p(A)=0}\neq\emptyset.
\end{equation*}
\(\calA\) is a \emph{maximal \(\theta\)-independent family} if, for all \(\theta\)-independent \(\calA'\supseteq\calA\) on \(X\), \(\calA'=\calA\).

\subsection{Ideals}

\(\calI\subseteq\power(X)\) is an \emph{ideal} (on \(X\)) if \(\calI\cap\Set{\emptyset,X}=\Set{\emptyset}\) and \(\calI\) is closed under subsets of elements and finite unions of elements. We say that \(\calI\) is: \emph{\(\lambda\)-complete} if \(\calI\) is closed under unions of fewer than \(\lambda\) many elements; \emph{\(\lambda\)-saturated} if, for all \({\Set{A_\alpha\mid\alpha<\lambda}\subseteq\power(X)\setminus\calI}\), there are \(\alpha<\beta<\lambda\) such that \(A_\alpha\cap A_\beta\notin\calI\); \emph{non-trivial} if \([X]^1\subseteq\calI\); and \emph{prime} if for all \(A\subseteq X\), \(A\in\calI\) or \(X\setminus A\in\calI\).

\(\calF\subseteq\power(X)\) is a \emph{filter} (on \(X\)) if \(\calF^\ast\defeq\Set{X\setminus A\mid A\in\calF}\) is an ideal on \(X\). We say that \(\calF\) is \emph{\(\lambda\)-complete} if \(\calF^\ast\) is \(\lambda\)-complete (as an ideal), and an \emph{ultrafilter} if \(\calF^\ast\) is prime. An ultrafilter \(\calU\) on a set \(S\subseteq\power(X)\) is \emph{fine} if for all \(x\in X\), \({\Set{A\in S\mid x\in A}\in\calU}\). We shall say that an ultrafilter \(\calU\) on a cardinal \(\kappa\) is a \emph{measure} if it is \(\kappa\)-complete, and that a measure \(\calU\) on \(\kappa\) is \emph{normal} if for all \(A\in\calU\) and \(f\in\prod A\) there is \(B\in\calU\) such that \(f\res B\) is constant.

Given two models \(V\subseteq W\) of \(\ZFC\) and an ideal \(\calI\in V\) on \(X\), we denote by \(\tup{\calI}^W\) the ideal generated by \(\calI\) in the extension:
\begin{equation*}
\tup{\calI}^W=\Set*{A\in\power(X)^W\mid(\exists Y\in\calI)\,A\subseteq Y}.
\end{equation*}

\subsection{Large cardinals}

We briefly revise measurable cardinals and the construction of ultrapower embeddings in order to fix notation, though our presentation is standard. Given a \(\sigma\)-complete ultrafilter \(\calU\) on an infinite set \(X\) and functions \(f,g\colon X\to V\) we say that \(f\mathrel{=_\calU}g\) if \(\Set{x\in X\mid f(x)=g(x)}\in\calU\), and denote by \([f]_\calU\) the \(=_\calU\)-equivalence class of \(f\). We then endow these classes with the relation \(\in_\calU\) given by \([f]_\calU\mathrel{\in_\calU}[g]_\calU\) if \({\Set{x\in X\mid f(x)\in g(x)}\in\calU}\). Finally, we identify this ultrapower construction
\begin{equation*}
\Ult(V,\calU)=\left(\Set*{[f]_\calU\mid f\colon X\to V},\in_\calU\right)
\end{equation*}
and the Mostowski collapse \(M\) of this structure, going as far as to say \(a=[f]_\calU\) to mean that \(a\) is the element of \(M\) associated with \([f]_\calU\) under the collapse. We then refer to the elementary embedding \(j_\calU\colon V\to M\) given by \(j_\calU(a)=[c_a]_\calU\), where \(c_a\) is the constant function \(X\to\Set{a}\), as the associated \emph{ultrapower embedding} of \(\calU\). The \emph{critical point} of an elementary embedding \(j\), denoted \(\crit(j)\) is \(\min\Set{\alpha\mid\alpha<j(\alpha)}\). If \(\calU\in V\) then \(\crit(j_\calU)\) is measurable in \(V\) (and hence a regular strong limit). We say that a transitive inner model \(M\) is \emph{\(\lambda\)-closed} to mean that \(M^{{<}\lambda}\subseteq M\). We may say that \(M\) is \(\lambda\)-closed \emph{in \(V\)} to emphasise that we specifically mean \(M^{{<}\lambda} \cap V \subseteq M \subseteq V\).

\subsubsection{The Mitchell order}

Given normal measures \(\calU\) and \(\calV\) on \(\kappa\), say that \(\calV\mitchless\calU\) if \(\calV \in \Ult(\calU,\calV)\). The relation \(\mitchless\) here is called the \emph{Mitchell order}. This order was introduced and proved to be well-founded by Mitchell in \cite{mitchell_sets_1974}, so we may therefore endow such measures with their \emph{Mitchell rank} \(o(\calU)\), the order type of \(\Set{\calV\mid\calV\mitchless\calU}\). Similarly we define the Mitchell rank of a cardinal \(\kappa\) to be the height of the tree induced by \(\mitchless\), denoted \(o(\kappa)\). In particular, our convention is that \(o(\kappa)>0\) if and only if \(\kappa\) is measurable. Also, \(o(\calU)=\alpha\) if and only if \({\Set{\lambda<\kappa\mid o(\lambda)=\alpha}\in\calU}\).\footnote{This is easiest to prove by noting  \(j_\calU(\tup{o(\alpha)\mid\alpha<\kappa})(\kappa)=o(\calU)\) (from \cite{mitchell_sets_1983}).}

\subsubsection{\(\theta\)-strongly compact cardinals}

We also require a large cardinal property that was introduced in \cite{bagaria_group_2014}.

\begin{defn}
For \(\theta\leq\kappa\) we say that \(\kappa\) is \emph{\(\theta\)-strongly compact} if every \(\kappa\)-complete filter on an arbitrary set \(X\) can be extended to a \(\theta\)-complete ultrafilter on \(X\).
\end{defn}

Note that sometimes in the literature a ``\(\theta\)-strongly compact'' cardinal refers to a cardinal \(\kappa\leq\theta\) such that there is a \(\kappa\)-complete fine ultrafilter on \(\power_\kappa(\theta)\). We shall not make use of this other definition.

\begin{thm}[{\cite[Theorem~4.7]{bagaria_group_2014}}]\label{thm:sc-characterisation}
The following are equivalent:
\begin{enumerate}[label=\textup{(\arabic*)}]
\item\label{cond:ultra-extend} \(\kappa\) is \(\theta\)-strongly compact.
\item\label{cond:embedding} For all \(\alpha\geq\kappa\) there is an elementary embedding \(j\colon V\to M\), where \(M\) is a transitive inner model of \(\ZFC\), such that \(\crit(j)\geq\theta\) and, for some \(D\in M\), \(j``\alpha\subseteq D\) and \(M\vDash\abs{D}<j(\kappa)\).
\item\label{cond:fine-ultra} For all \(\alpha\geq\kappa\) there is a fine \(\theta\)-complete ultrafilter on \(\power_\kappa(\alpha)\).
\end{enumerate}
\end{thm}

\begin{factstar}
If \(\calU\) is a fine \(\theta\)-complete ultrafilter on \(\power_\kappa(\alpha)\) then \(j_\calU\) satisfies \cref{cond:embedding} with \(D=[\id]_\calU\).
\end{factstar}

\cref{cor:immediate-sc}, remarked upon in \cite{bagaria_group_2014}, is immediate but important.

\begin{cor}\label{cor:immediate-sc}
Let \(\kappa\) be \(\theta\)-strongly compact.
\begin{enumerate}[label=\textup{(\roman*)}]
\item\label{cor:im-sc;cond:measure} There is a measurable cardinal \(\mu\) such that \(\theta\leq\mu\leq\kappa\). If \(\mu\) is the least measurable cardinal greater than or equal to \(\theta\) then \(\kappa\) is \(\mu\)-strongly compact.
\item\label{cor:im-sc;cond:upwards} For all \(\lambda\geq\kappa\), \(\lambda\) is \(\theta\)-strongly compact. In particular, there is a \(\theta\)-strongly compact cardinal \(\lambda\) such that \(2^{{<}\lambda}=\lambda\).
\end{enumerate}
\end{cor}

\begin{proof}
\textit{Item}\enspace{}\ref*{cor:im-sc;cond:measure}.\quad{}Firstly, by \cref{cond:embedding} in \cref{thm:sc-characterisation} with \(\alpha=\kappa^+\), we obtain \(j\colon V\to M\) with \(\theta\leq\crit(j)\). Since there is an injection \(D\to j(\kappa)\) in \(M\) with \(j``\kappa^+\subseteq D\), we must have \(\kappa^+\leq j(\kappa)\), so \(\crit(j)\leq\kappa\). Hence \(\crit(j)\) is measurable with \({\theta\leq\crit(j)\leq\kappa}\).

Let \(\mu\) be the least measurable cardinal with \(\mu\geq\theta\). Then for each \(\alpha\geq\kappa\) there is \(j\colon V\to M\) such that \(\crit(j)\geq\theta\) and, for some \(D\in M\), \(j``\alpha\subseteq D\) and \(M\vDash\abs{D}<j(\kappa)\). However, in this case we must also have \(\crit(j)\geq\mu\). Hence \(\kappa\) is \(\mu\)-strongly compact.

\textit{Item}\enspace{}\ref*{cor:im-sc;cond:upwards}.\quad{}For all \(\lambda\geq\kappa\), any \(\lambda\)-complete filter \(\calF\) on any set \(X\) is also \(\kappa\)-complete. Hence, by the definition of \(\kappa\) being \(\theta\)-strongly compact, the filter \(\calF\) extends to a \(\theta\)-complete ultrafilter \(\calU\) on \(X\). Setting \(\lambda=\beth_\omega(\kappa)=\sup\Set{\kappa,\,2^\kappa,\,2^{2^{\kappa}},\dots}>\kappa\) we have \(2^{{<}\lambda}=\lambda\) and \(\lambda\) is \(\theta\)-strongly compact.
\end{proof}

\subsection{Forcing}

Our presentation of forcing is standard. For a more in-depth introduction to forcing there are a wealth of excellent texts, such as \cite[Ch.~14]{jech_set_2003}. By a \emph{notion of forcing} we mean a partial order \(\tup{\bbP,\leq_\bbP}\) with a maximum \(\1_{\bbP}\), though subscripts will be omitted when clear from context. We force downwards, so if \(q\leq p\) then we say that \(q\) is an \emph{extension} or \emph{stronger condition} than \(p\). \(\bbP\) has the \emph{\(\kappa\)-chain condition}, or \(\kappa\)-c.c., if every antichain in \(\bbP\) has cardinality less than \(\kappa\). We say that \(\bbP\) is \emph{\(\kappa\)-closed} if for every \(\gamma<\kappa\) and descending chain \({\Set{p_\alpha\mid\alpha<\gamma}}\subseteq\bbP\) there is \(q\in\bbP\) such that \(q\leq p_\alpha\) for all \(\alpha<\gamma\). By a \emph{\(V\)-generic filter} for \(\bbP\) we mean a filter \(G\subseteq\bbP\) such that for all dense \(D\subseteq\bbP\) with \(D\in V\), \(G\cap D\neq\emptyset\).

We say that \(p,q\in\bbP\) are \emph{compatible}, denoted \(p\comp_{\bbP}q\), if there is \(r\in\bbP\) with \(r\leq p,q\). Otherwise, we say that \(p\) and \(q\) are \emph{incompatible}, denoted \(p\perp_{\bbP}q\). Again, subscripts are omitted when clear from context.

Given a set of \(\bbP\)-names \(X\) we define the \emph{bullet name} of \(X\), denoted \(X^\bullet\), to be \(\Set{\tup{\1_{\bbP},\ddx}\mid \ddx\in X}\). Then \(X^\bullet\) is always interpreted as the set \({\Set{\ddx^G\mid \ddx\in X}}\) in the extension \(V[G]\). This notation extends to tuples, functions, etc.\ with ground model domains. For \(X\in V\) we define the \emph{check name} of \(X\), denoted \(\check{X}\), to be \({\Set{\check{x}\mid x\in X}^\bullet}\). That is, \(\check{X}^G=X\in V[G]\) for all \(V\)-generic \(G\subseteq\bbP\). We may alternatively use the bullet notation to define a canonical name of a definable object. For example, \(\power(\check{X})^\bullet\) is the canonical name for the power set of \(X\) in the forcing extension.

Given non-empty sets \(A\) and \(B\) we denote by \(\Add(B,A)\) the notion of forcing given by partial functions \(p\colon A\times B\to2\) such that \(\abs{p}<\abs{B}\).\footnote{When \(\abs{A}\geq\abs{B}\) we denote by \(\Fn(A,2,B)\) the notion of forcing with conditions that are partial functions \(p\colon A\to2\) such that \(\abs{p}<\abs{B}\). In this case \(\Fn(A,2,B)\) is isomorphic to \(\Add(B,A)\). While \(\Fn(A,2,B)\) is not used in this paper, it is used in \cite{kunen_maximal_1983}, to which we sometimes refer.} Note that if \(\kappa\) is a cardinal and \(A\) is non-empty then \(\Add(\kappa,A)\) is \(\cf(\kappa)\)-closed and has the \((\kappa^{{<}\kappa})^+\)-c.c.

\begin{defn}
Given a forcing iteration \(\tup{\bbP_\alpha,\ddbbQ_\alpha\mid\alpha<\delta}\) of limit length \(\delta\), the \emph{inverse limit} of the system is the notion of forcing \(\bbP\) with conditions given by functions \(p\) with domain \(\delta\) such that, for all \(\gamma<\delta\), \(p\res\gamma\in\bbP_\gamma\). The ordering is given by \(q\leq p\) if for all \(\gamma<\delta\), \(q\res\gamma\leq_\gamma p\res\gamma\). The \emph{direct limit} of the system is the notion of forcing \(\bbP\) with conditions given by functions \(p\) with domain \(\gamma\) for some \(\gamma<\delta\) such that \(p\in\bbP_\gamma\). The ordering is given by \(q\leq p\) if \(\gamma=\dom(p)\leq\dom(q)\) and \(q\res\gamma\leq_\gamma p\). By an \emph{Easton support iteration} we mean that for all limit ordinals \(\alpha\), \(\bbP_\alpha\) is taken as an inverse limit if \(\alpha\) is singular, and a direct limit otherwise.
\end{defn}

\begin{defn}
Given two notions of forcing \(\bbP\) and \(\bbQ\), a function \(\psi\colon\bbP\to\bbQ\) is a \emph{dense embedding} if:
\begin{enumerate}[label=\textup{(\roman*)}]
\item For all \(p,p'\in\bbP\), \(p\leq_\bbP p'\) if and only if \(\psi(p)\leq_\bbQ\psi(p')\); and
\item for all \(q\in\bbQ\) there is \(p\in\bbP\) such that \(\psi(p)\leq_\bbQ q\).
\end{enumerate}
Note that in this case we also have \(p\perp_\bbP p'\) if and only if \(\psi(p)\perp_\bbQ\psi(p')\).
\end{defn}

\subsection{Closure points and elementary embeddings}

In \cref{s:applications;ss:class-of-measurables} we shall wish to consider which elementary embeddings are found in a forcing extension. In our instance we are making use of an Easton-support iteration that, as is often the case for Easton-support iterations, admits no elementary embeddings that do not extend a ground-model embedding. To show this, we will make use of \emph{closure points}, as explored in \cite{hamkins_extensions_2003}.

\begin{defn}
A notion of forcing has a \emph{closure point} at \(\delta\) if it can be factored as \(\bbP_0\iter\ddbbP_1\), where \(\bbP_0\) is atomless, \(\abs{\bbP_0}\leq\delta\), and \(\1\forces_{\bbP_0}``\ddbbP_1\) is \({\leq}\delta\) strategically closed''.\footnote{We shall not define strategic closure here. A \(\delta^+\)-closed notion of forcing is \({\leq}\delta\) strategically closed, and we will only ever use this case. A definition can be found in \cite[Ch.~12, Definition~5.15]{handbook}.}
\end{defn}

The following result is a combination of Lemma~13 and Theorem~10 in \cite{hamkins_extensions_2003}. While Lemma~13 and Theorem~10 in \cite{hamkins_extensions_2003} are more powerful than what we present here, the \cref{prop:hamkins-closure-gives-unlifting} is all that we will need.

\begin{prop}[Hamkins]\label{prop:hamkins-closure-gives-unlifting}
If \(\bbP\) has a closure point at \(\delta\), \(G\subseteq\bbP\) is \(V\)-generic, and \(\calU \in V[G]\) is a normal measure on \(\kappa>\delta\), then \(\calU \cap V \in V\) is a normal measure on \(\kappa\).
\end{prop}

\section{Hammers}\label{s:tools}

We have two main tools at our disposal that work together to produce models in which there is a \(\theta\)-independent family on a set \(X\). Let us begin with \cref{lem:ideal-gives-family}, a method for showing that maximal \(\theta\)-independent families exist contingent on the presence of a particular ideal \(\calI\). The proof and result are essentially due to Kunen in the form of \cite[Lemma~2.1]{kunen_maximal_1983}, but we have softened the requirements.

\begin{lem}\label{lem:ideal-gives-family}
Let \(\theta\) be a regular uncountable cardinal, \(X\) a set with \(\abs{X}\geq\theta\), and \(\calI\) a \(\theta\)-complete ideal over \(X\) such that \(\Add(\theta,2^X)\) densely embeds into \(\power(X)/\calI\). Then there is a maximal \(\theta\)-independent family \(\calA\subseteq\power(X)\).
\end{lem}

\begin{proof}
Let \(\bbP=\Add(\theta,2^X)\) and \(\psi\colon\bbP\to\power(X)/\calI\) be a dense embedding. For all \(f\in2^X\) and \(\delta<\theta\) choose \(A_{f,\delta}\subseteq X\) such that \([A_{f,\delta}]=\psi(\stfd1)\) and define \(\vphi\colon\bbP\to\power(X)\) by
\begin{equation}\label{eqn:vphi-p}
\vphi(p)=\bigcap\Set*{A_{f,\delta}\mid p(f,\delta)=1}\cap\bigcap\Set*{X\setminus A_{f,\delta}\mid p(f,\delta)=0}.\tag{\ensuremath{*}}
\end{equation}
\begin{claim}\label{claim:psi-is-vphi}
For all \(p\in\bbP\), \(\psi(p)=[\vphi(p)]\).
\end{claim}
\begin{poc}
We shall first show that \(\psi(\stfd0)=[X\setminus A_{f,\delta}]\). Let \(B_{f,\delta}\) be chosen so that \(\psi(\stfd0)=[B_{f,\delta}]\). Since \(\stfd0\perp\stfd1\), \(A_{f,\delta}\cap B_{f,\delta}\in\calI\). For all \(A\notin\calI\) there is \(p\in\bbP\) such that \(\psi(p)\leq[A]\), and either \(p\comp\stfd0\) or \(p\comp\stfd1\). Hence, setting \([C]=\psi(p)\), we have that \(C\cap A_{f,\delta}\notin\calI\) or \(C\cap B_{f,\delta}\notin\calI\). In particular, letting \(D=X\setminus(A_{f,\delta}\cup B_{f,\delta})\), we have \(D\cap A_{f,\delta}=D\cap B_{f,\delta}=\emptyset\in\calI\), so \(D\in\calI\). That is, \([A_{f,\delta}\cup B_{f,\delta}]=[X]\) and so \([B_{f,\delta}]=[X\setminus A_{f,\delta}]\) as required.

Therefore, for all \(p\in\bbP\), \(\psi(p)\leq[A_{f,\delta}]\) whenever \(p(f,\delta)=1\) and \(\psi(p)\leq[X\setminus A_{f,\delta}]\) whenever \(p(f,\delta)=0\). Given that \(\abs{p}<\theta\) and \(\calI\) is \(\theta\)-complete, this means that \(\psi(p)\leq[\vphi(p)]\). Setting \(\psi(p)=[A]\), if \([A]<[\vphi(p)]\) then \(\vphi(p)\setminus A\notin\calI\) and so there is \(q\in\bbP\) such that \(\psi(q)\leq[\vphi(p)\setminus A]\). In particular, \(\psi(q)\leq[A_{f,\delta}]=\psi(\stfd1)\) whenever \(p(f,\delta)=1\) and \(\psi(q)\leq[X\setminus A_{f,\delta}]=\psi(\stfd0)\) whenever \(p(f,\delta)=0\). That is, \(q\leq p\) and thus \(\psi(q)\leq[A]\). However, this cannot be the case since \([\vphi(p)\setminus A]\perp[A]\).
\end{poc}
Let \(\calA=\Set{A_{f,\delta}\mid f\in2^X,\,\delta<\theta}\). Then for all \(p\in\bbP\), \([\vphi(p)]=\psi(p)\neq[\emptyset]\), so \(\vphi(p)\notin\calI\) and thus \(\calA\) is \(\theta\)-independent. Furthermore, for all \(A\notin\calI\) there is \(p\in\bbP\) such that \(\psi(p)=[\vphi(p)]\leq[A]\), and thus for all \(A\subseteq X\) there is \(p\in\bbP\) such that \([\vphi(p)]\leq[A]\) or \([\vphi(p)]\leq[X\setminus A]\). However, this is not quite true maximality, as we would require that for all \(A\subseteq X\) there is \(p\in\bbP\) such that \(\vphi(p)\subseteq A\) or \(\vphi(p)\subseteq X\setminus A\). To achieve this we alter \(\calA\) slightly.

Enumerate \(\calI\) as \(\Set{C_f\mid f\in2^X}\) (with repeat entries if necessary) and define \(A_{f,\delta}'=A_{f,\delta}\setminus C_f\). Since the representatives \(A_{f,\delta}\in\psi(\stfd1)\) were chosen arbitrarily, \cref{claim:psi-is-vphi} still holds for \(\calA'=\Set{A_{f,\delta}'\mid f\in2^X,\,\delta<\theta}\), where we define \(\vphi'\) analogously to \(\vphi\) in \cref{eqn:vphi-p}. Hence, if \(A\notin\calI\) then there is \(p\in\bbP\) such that \([\vphi'(p)]\leq[A]\), so \(\vphi'(p)\setminus A=C_f\in\calI\). Since \(\abs{p}<\theta\) there is \(\delta<\theta\) such that \(\tup{f,\delta}\notin\dom(p)\), and hence \(\vphi'(p\cup\stfd1)\subseteq\vphi'(p)\setminus C_f\subseteq A\) as required.
\end{proof}

\begin{rk}
The statement of \cref{lem:ideal-gives-family} is, on the surface, a strengthening of Kunen's result, as we have removed two requirements (both \(2^{{<}\theta}=\theta\) and that \(\calI\) is \(\theta^+\)-saturated) and weakened further requirements (we only need \(\calI\) to be \(\theta\)\nobreakdash-complete, rather than \(\abs{X}\)-complete, and only demand that there is a dense embedding of \(\Add(\theta,2^X)\) into \(\power(X)/\calI\), rather than an isomorphism). This weakening is partially illusory. The proof of \cref{lem:ideal-gives-family} in \cite{kunen_maximal_1983} makes little use of some of these extraneous assumptions, and some of these requirements that we have altered are consequences: By \cref{thm:family-gives-ideal} it will be the case that \(2^{{<}\theta}=\theta\) and, since \(\Add(\theta,2^X)\) densely embeds into \(\power(X)/\calI\), we recover that \(\calI\) is \(\theta^+\)-saturated by the chain condition.
\end{rk}

Our second tool is an old technique present in \cite{kunen_maximal_1983} (among many other places) for obtaining ideals \(\calI\) on \(X\) such that \(\power(X)/\calI\) is a complete Boolean algebra isomorphic to a desired notion of forcing. This method is closely tied to the idea of lifting elementary embeddings: If \(j\colon V\to M\) is an elementary embedding and \(G\) is \(V\)-generic, then one can lift the elementary embedding to \(\hj\colon V[G]\to M[j(G)]\), where \(j(G)\) is an appropriate \(M\)-generic filter. However, if \(j\) was definable in \(V\) and \(j(G)\notin V[G]\), then we may be unable to lift the embedding definably in \(V[G]\). \cref{thm:duality-prime} can be understood intuitively as the idea that if \(j=j_\calU\) is an ultrapower embedding and \(G\subseteq\bbP\) is \(V\)-generic, then \(\calI=\tup{\calU^\ast}^{V[G]}\) will be a prime ideal only if \(j\) lifts to \(V[G]\), and if it does not then \(\power(X)/\calI\) is the extra amount of forcing required to lift the embedding: \(\bbP\iter\power(X)/\calI\cong j(\bbP)\).

This technique has been the subject of much refinement, culminating in Foreman's \emph{Duality Theorem}, from \cite{foreman_calculating_2013}, which bring precipitous ideals and a more refined definition of \(\calI\) into the fold. We do not need quite the level of complexity that the Duality Theorem affords, and so we shall present a specialised version that is localised to prime ideals, the scenario that we have described. For the case that \(\ddbbR\) is forced to be trivial, one can follow the technique of \cite{kunen_maximal_1983} in which the special case of \(\bbP=\Add(\omega_1,\kappa)\) was applied\footnote{Rather, the special case \(\bbP=\Fn(\kappa,2,\omega_1)\), but these are isomorphic.} to see a proof.

\begin{thm}\label{thm:duality-prime}
Let \(\calU\) be a \(\sigma\)-complete ultrafilter on a set \(X\) with ultrapower embedding \(j=j_\calU\colon V\to M\). Let \(\bbP\) be a notion of forcing such that \(j(\bbP)\cong\bbP\iter\ddbbQ\iter\ddbbR\) by an isomorphism \(\pi\) satisfying \(\pi(j(p))=\tup{p,\1,\1}\). Suppose that, for all \(V\)-generic \(G\iter H\subseteq\bbP\iter\ddbbQ\), there is \(M[G\iter H]\)-generic \(F\subseteq\ddbbR^{G\iter H}\) with \(F\in V[G\iter H]\). Then there is a \(\bbP\)-name for an ideal \(\dot{\calI}\) on \(X\) such that \(\1\forces_{\bbP}\power(\check{X})^\bullet/\dot{\calI}\cong B(\ddbbQ)\).

In the special case that \(\ddbbR\) is forced to be the trivial forcing, \(\calI\) is in fact \(\tup{\calU^\ast}^{V[G]}\), the ideal generated by \(\calU^\ast\) in the extension. Hence, if \(\bbP\) is \(\theta\)-distributive and \(\calU\) is \(\theta\)-complete then \(\calI\) will be \(\theta\)-complete as well.
\end{thm}

\section{Nails}\label{s:applications}

Let us now apply our tools to produce examples of maximal \(\theta\)-independent families, beginning with \cref{thm:kunens-theorem-applicable}. Our presentation of this result is a slight extension of Kunen's original method to allow for general regular uncountable \(\theta\). 

\begin{thm}[{\cite[Theorem~2]{kunen_maximal_1983}}]\label{thm:kunens-theorem-applicable}
Let \(\kappa\) be a measurable cardinal, \(\theta<\kappa\) be uncountable and regular, and \(G\) be \(V\)-generic for \(\Add(\theta,\kappa)\). Then there is a maximal \(\theta\)-independent family \(\calA\subseteq\power(2^\theta)\) in \(V[G]\).
\end{thm}

\begin{proof}
Let \(\kappa\) have measure \(\calU\) and ultrapower embedding \(j=j_\calU\colon V\to M\), so \(\crit(j)=\kappa\). Then in \(V[G]\) we have \(\theta^{{<}\theta}=\theta\), \(2^\theta=\kappa\), and
\begin{align*}
j(\Add(\theta,\kappa))&=\Add(\theta,j(\kappa))\\
&\cong\Add(\theta,j``\kappa)\times\Add(\theta,j(\kappa)\setminus j``\kappa)\\
&\cong\Add(\theta,\kappa)\times\Add(\theta,j(\kappa)\setminus\kappa)\times\Set{\1}.
\end{align*}
Since each \(p\in\Add(\theta,\kappa)\) has cardinality less than \(\theta\) we have that \({j(p)=j``p=p}\). Furthermore, since \(M\) is \(\kappa^+\)-closed, \(\Add(\theta,j(\kappa))^M=\Add(\theta,j(\kappa))^V\). By \cref{thm:duality-prime}, setting \({\calI=\tup{\calU^\ast}^{V[G]}}\), we have \(\power(\kappa)/\calI\cong B(\Add(\theta,j(\kappa)\setminus\kappa))\). Note here that since \(\Add(\theta,\kappa)\) is \(\theta\)-closed, \(\Add(\theta,X)^V=\Add(\theta,X)^{V[G]}\) for all \(X\in V\) and so the isomorphism class of \(\Add(\theta,X)\) (in either \(V\) or \(V[G]\)) depends only on the cardinality of \(X\). Furthermore, since \(\Add(\theta,\kappa)\) is \(\theta\)-closed, \(\calI\) is \(\theta\)-complete.

Finally, since \(\kappa\) is measurable, \(2^\kappa<j(\kappa)<(2^\kappa)^+\) and so \(\Add(\theta,j(\kappa)\setminus\kappa)\) is isomorphic to \(\Add(\theta,2^\kappa)\). Therefore, by \cref{lem:ideal-gives-family}, there is a maximal \(\theta\)-independent family \(\calA\subseteq\power(\kappa)=\power(2^\theta)\) in \(V[G]\).
\end{proof}

Hence, if \(\ZFC\) plus the existence of a measurable cardinal is consistent, then so is \(\ZFC\) plus the existence of a maximal \(\sigma\)-independent family on \(2^{\omega_1}\). Furthering this, Kunen recovers the consistency of a measurable cardinal from the consistency of a maximal \(\theta\)-independent family.

\begin{thm}[{\cite[Theorem~1]{kunen_maximal_1983}}]\label{thm:family-gives-ideal}
Let \(\theta\) be an uncountable regular cardinal such that there is a maximal \(\theta\)-independent family \(\calA\subseteq\power(\lambda)\). Then \(2^{{<}\theta}=\theta\) and, for some \(\kappa\) such that \(\sup\Set{(2^\alpha)^+\mid\alpha<\theta}\leq\kappa\leq\min\Set{\lambda,2^\theta}\), there is a non-trivial \(\theta^+\)-saturated \(\kappa\)-complete ideal over \(\kappa\).
\end{thm}

The full proof may be found in \cite{kunen_maximal_1983}, with the roles of \(\kappa\) and \(\lambda\) swapped, but we shall sketch it here.

\begin{proof}[Sketch proof]
We say that a maximal \(\theta\)-independent family \(\calA\subseteq\power(\lambda)\) is \emph{globally maximal} if, setting \(P\) to be the set of partial functions \(p\colon\calA\to2\) with \(\abs{p}<\theta\),
\begin{equation*}
(\forall p\in P)(\forall X\subseteq\calA^p)(\exists q\supseteq p)(\calA^q\subseteq X \lor \calA^q \cap X = \emptyset).
\end{equation*}
\begin{factstar}
There is \(p\in P\) such that \(\calA/p=\Set{A \cap \calA^p \mid A \in \calA \setminus \dom(p)}\) is globally maximal \(\theta\)-independent.
\end{factstar}
Hence, replacing \(\lambda\) by some \(\lambda'<\lambda\) if necessary, we may assume that \(\calA\) is globally maximal \(\theta\)-independent on \(\lambda\). Let
\begin{equation*}
\calI_{\calA}=\Set*{X\subseteq\lambda\mid(\forall p\in P)\calA^p\nsubseteq X}.
\end{equation*}
Then \(\calI_{\calA}\) is \(\theta^+\)-saturated and \((2^\alpha)^+\)-complete for all \(\alpha<\theta\). Setting \(\kappa\) be least such that \(\calI_{\calA}\) is not \(\kappa^+\)-complete, we can refine \(\calI_{\calA}\) to a \(\kappa\)-complete \(\theta^+\)-saturated ideal on \(\kappa\). We immediately have \(\sup\Set{(2^\alpha)^+\mid\alpha<\theta}\leq\kappa\leq\lambda\). On the other hand, if \(\calA_0\in[\calA]^\theta\) then \(P_0=\Set{\calA^p\mid p\text{ a total function }\calA_0\to2}\subseteq\calI_{\calA}\), but \(\bigcup P_0=\lambda\notin\calI_\calA\) and so \(\kappa\leq2^\theta\).
\end{proof}

\begin{rk}
The maximal \(\theta\)-independent families constructed by \cref{lem:ideal-gives-family} are globally maximal \(\theta\)-independent. Furthermore, when we later construct maximal \(\kappa\)-independent families on \(\kappa\), the bounds directly give us that there is a \(\kappa\)-complete and \(\kappa^+\)-saturated ideal on \(\kappa\).
\end{rk}

\begin{cor}
If there is a maximal \(\theta\)-independent family for some uncountable regular \(\theta\) then there is an inner model containing a measurable cardinal.
\end{cor}

\begin{proof}
By \cite[Theorem~11.13]{kunen_some_1970}, if \(\kappa\) carries a \(\kappa\)-complete, \(\kappa^+\)-saturated ideal then there is an inner model in which \(\kappa\) is measurable. Since \(\kappa\geq\theta\), \(\theta^+\)-saturated implies \(\kappa^+\)-saturated.
\end{proof}

In fact, \cite[Theorem~11.13]{kunen_some_1970} can be extended to any finite collection of saturated ideals on increasing cardinals, as noted in \cite{schlutzenberg_mo_2022}.

\begin{lem}[{\cite{schlutzenberg_mo_2022}}]\label{prop:ideals-give-measurables}
If \(\kappa_0<\cdots<\kappa_{n-1}\) are uncountable regular cardinals such that for all \(i < n\) there is a normal \(\kappa_i\)-complete, \(\kappa_i^+\)-saturated ideal \(\calI_i \subseteq \power(\kappa_i)\), then in \(L[\calI_0,\dots,\calI_{n-1}]\), \(\kappa_i\) is measurable for all \(i\).
\end{lem}

It follows from \cref{thm:family-gives-ideal} and \cref{prop:ideals-give-measurables} that if there is a maximal \(\theta\)-independent family on \(\lambda\) and a maximal \(\theta'\)-independent family on \(\lambda'\) such that \({\min\Set{\lambda,2^\theta}<\sup\Set{(2^\alpha)^+\mid\alpha<\theta'}}\), then there is an inner model with two measurable cardinals, and indeed this pattern holds for all finite collections of such families. Therefore, a corollary of \cref{thm:main-sc} is that the consistency of an \(\aleph_1\)-strongly compact cardinal implies the consistency of any finite number of measurable cardinals (though this is already known).

Kunen briefly sketches how to obtain a maximal \(\kappa\)-independent family on inaccessible \(\kappa\), starting with a model in which \(\kappa\) is measurable. This requires a slightly more delicate use of \cref{thm:duality-prime} to obtain the result. We have also included additional content regarding lifting normal measures, which will be useful when proving \cref{thm:measurable-class-families}.

\begin{prop}\label{prop:iteration-length-kappa}
Let \(\kappa\) be measurable with normal measure \(\calU\), \(2^\kappa=\kappa^+\), and \(A\in\calU\) be a set of regular cardinals. Let \(G\) be \(V\)-generic for the Easton-support iteration \({\bbP=\bigast_{\alpha\in A}\Add(\alpha,\alpha^+)}\). Then in \(V[G]\) there is a maximal \(\kappa\)-independent family \(\calA\subseteq\power(\kappa)\). Furthermore, if \(\calV\in V\) is a normal measure on \(\kappa\) such that \(A\notin\calV\) then there is a normal measure \(\hat\calV\supseteq\calV\) on \(\kappa\) in \(V[G]\).
\end{prop}

\begin{proof}
Let \(j=j_\calU\colon V\to M\), and \({H\subseteq\Add(\kappa,\kappa^+)}\) be \(V[G]\)-generic. Note that \(j(\kappa)<(2^\kappa)^+=\kappa^{++}\). Furthermore, \(M\) is \(\kappa^+\)-closed, and this is preserved by the forcing (\(\bbP\iter\Add(\kappa,\kappa^+)\in M\)), so \(M[G\iter H]\) is also \(\kappa^+\)-closed.

Let \(\ddbbR=j(\bbP)/(\bbP\iter\Add(\kappa,\kappa^+)^\bullet)\), noting that due to the Easton support we truly have \(j(\bbP)\cong\bbP\iter\Add(\kappa,\kappa^+)\iter\ddbbR\) as required in \cref{thm:duality-prime}. Let
\begin{equation*}
\bbR=\ddbbR^{G\iter H}=\bigast_{\alpha\in j(A)\setminus\kappa^+}\Add(\alpha,(\alpha^+)^M)^M.
\end{equation*}
\(\abs{\bbP}^V=\kappa\), so \(\abs{j(\bbP)}^M=j(\kappa)\), and hence \(\abs{\bbR}^{V[G]}=\kappa^+\). Furthermore, each iterand of \(\bbR\) is \(\alpha\)-closed according to \(M\) for some \(\alpha\geq\kappa^+\). Since \(M[G\iter H]\) is \(\kappa^+\)-closed, this means that each iterand of \(\bbR\) is \(\kappa^+\)-closed (in \(V[G\iter H]\)) and, since it is an iteration of length \(j(\kappa)\geq\kappa^+\), \(\bbR\) itself is \(\kappa^+\)-closed. However, \(\bbP\) has only \(\kappa\)-many maximal antichains: Each iterand is of cardinality less than \(\kappa\) and so has fewer than \(\kappa\)-many antichains. Hence \(M[G\iter H]\vDash``\bbR\) has only \(j(\kappa)\)-many maximal antichains''. Since \(j(\kappa)<\kappa^{++}\) we can build an \(M[G\iter H]\)-generic filter \(F\subseteq\bbR\) in \(V[G\iter H]\). Hence, by \cref{thm:duality-prime}, in \(V[G]\) there is an ideal \(\calI\) on \(\kappa\) such that \(\power(\kappa)/\calI\cong B(\Add(\kappa,\kappa^+))^{V[G]}\). This ideal can be expressed as
\begin{equation*}
\calI\defeq\Set*{\ddA^G\subseteq\kappa\mid\1\forces_{\bbP\iter\Add(\kappa,\kappa^+)\iter\ddbbR/\ddF}\check{\kappa}\notin j(\ddA)},
\end{equation*}
where \(\ddF\) is a \(\bbP\iter\Add(\kappa,\kappa^+)\)-name for an \(M[G\iter H]\)-generic ideal \(F\subseteq\bbR\). In this case, if \({\Set{\ddA_\alpha\mid\alpha<\gamma}\subseteq\calI}\) for some \(\gamma<\kappa\) then, since \(\crit(j)=\kappa\), \(j(\bigcup\ddA_\alpha)=\bigcup j(\ddA_\alpha)\), and so \(\bigcup\ddA_\alpha^G\in\calI\). Hence, \(\calI\) is \(\kappa\)-complete as required, and so by \cref{lem:ideal-gives-family} there is a maximal \(\kappa\)-independent family on \(\kappa\) in \(V[G]\).

On the other hand, let \(\calV \in V\) be a normal measure on \(\kappa\) such that \(A\notin\calV\), with ultrapower embedding \(i\colon V\to N\). Then \(i(\bbP)\cong\bbP\iter\ddbbR\) (without the \(\Add(\kappa,\kappa^+)\) iterand), and so \(F\in V[G]\). Hence we may lift \(i\) to \(\hi\colon V[G]\to M[G\iter F]\) in \(V[G]\) and obtain normal measure \(\hat\calV=\Set{B\subseteq\kappa\mid\kappa\in\hi(B)}\) on \(\kappa\) extending \(\calV\).
\end{proof}

\subsection{A \(\theta^+\)-strongly compact cardinal}\label{s:applications;ss:strongly-compact}

These techniques are ripe for transfer to other large cardinal properties. In the following we shall find that \(\kappa\) being \(\theta^+\)-strongly compact\footnote{Note that, for fixed \(\theta\), there is \(\mu>\theta\) such that \(\kappa\) is \(\mu\)-strongly compact if and only if \(\kappa\) is \(\theta^+\)-strongly compact.} for uncountable regular \(\theta\) is sufficient to produce the ultrapower embeddings \(j\colon V\to M\) that give rise to a proper class of \(\lambda\) such that there is a maximal \(\theta\)-independent families \(\calA\subseteq\power(\lambda)\). The transfer is not entirely clean, as we additionally require that \(2^{{<}\kappa}=\kappa\), but as noted in \cref{cor:immediate-sc} this does not increase the consistency strength of the assumption.

\begin{thmalph}\label{thm:main-sc}
Let \(\kappa\) be \(\theta^+\)-strongly compact for some uncountable regular \(\theta<\kappa\), with \(2^{{<}\kappa}=\kappa\), and let \(G\) be \(V\)-generic for \(\Add(\theta,\kappa)\). In \(V[G]\), for all \(\lambda\geq\kappa\) with \(\cf(\lambda)\geq\kappa\), there is a maximal \(\theta\)-independent family \(\calA\subseteq\power(\lambda)\).
\end{thmalph}

\begin{proof}
Let \(\lambda\) be such that \(\cf(\lambda)\geq\kappa\). We wish to use \cref{lem:ideal-gives-family} to show that there is a maximal \(\theta\)-independent family \(\calA\subseteq\power(X)\), where \(X=\power_\kappa(\lambda)^V\) (noting that \(\abs{X}=\lambda\)). We therefore require a \(\theta\)-complete ideal \(\calI\) over \(X\) such that \(B(\Add(\theta,2^X))\) is isomorphic to \(\power(X)/\calI\) in \(V[G]\), which we shall obtain through \cref{thm:duality-prime}.

Let \(\calU\in V\) be a fine \(\theta\)\nobreakdash-complete ultrafilter on \(X\) and \(j=j_\calU\colon V\to M\). Since \(\kappa\geq\crit(j)>\theta\),
\begin{align*}
j(\Add(\theta,\kappa))&=\Add(\theta,j(\kappa))\\
&\cong\Add(\theta,j``\kappa)\times\Add(\theta,j(\kappa)\setminus j``\kappa)\\
&\cong\Add(\theta,\kappa)\times\Add(\theta,j(\kappa)\setminus\kappa)\times\Set{\1}.
\end{align*}
Furthermore, each \(p\in\Add(\theta,\kappa)\) is of cardinality less than \(\theta\) and thus \(j(p)=j``p\), so the isomorphism extends \(j(p)\mapsto\tup{p,\1,\1}\) as required. Hence, setting \(\calI=\tup{\calU^\ast}^{V[G]}\), we have \(B(\Add(\theta,j(\kappa)\setminus\kappa)^V)\cong\power(X)/\calI\) in \(V[G]\) by \cref{thm:duality-prime}. To finish we therefore need only show that
\begin{equation*}
\Add(\theta,j(\kappa)\setminus\kappa)^V\cong\Add(\theta,2^X)^{V[G]}.
\end{equation*}
\(\Add(\theta,\kappa)\) is \(\theta\)-closed so, for all \(Y\in V\), \(\Add(\theta,Y)^V=\Add(\theta,Y)^{V[G]}\) and so it is sufficient to prove that \(\abs{j(\kappa)\setminus\kappa}=\abs{(2^\lambda)^{V[G]}}\).

\(\Add(\theta,\kappa)\) is \((\theta^{{<}\theta})^+\)-c.c.\ and \(\theta^{{<}\theta}\leq\theta^{{<}\crit(j)}=\crit(j)\leq\kappa\), so \(\Add(\theta,\kappa)\) is \(\kappa^+\)-c.c. By standard techniques,\footnote{One could adapt the proof of \cite[Lemma~15.1]{jech_set_2003} to incorporate chain conditions, for example.} \({(2^\lambda)^{V[G]}\leq(\abs{\Add(\theta,\kappa)}^{\kappa\times\lambda})^V}\). \({\abs{\Add(\theta,\kappa)}\leq\kappa^\theta\leq\lambda^\lambda}\), so we get that \(\abs{(2^\lambda)^{V[G]}}\leq\abs{(2^\lambda)^V}\). Certainly \((2^\lambda)^V\subseteq(2^\lambda)^{V[G]}\) so we conclude that \(\abs{(2^\lambda)^V}=\abs{(2^\lambda)^{V[G]}}\). It is therefore sufficient to show that \(\abs{2^\lambda}=\abs{j(\kappa)\setminus\kappa}\) in \(V\). To that end, we work in \(V\) for the remainder of the proof.

Since \(2^\lambda>\kappa\) it is sufficient to show that \(2^\lambda\leq j(\kappa)<(2^\lambda)^+\). Let \(D=[\id]_\calU\) in \(M\). By the fineness of \(\calU\), \(j``\lambda\subseteq D\) and \(M\vDash\abs{D}<j(\kappa)\). By elementarity,
\begin{equation*}
M\vDash(\forall\gamma<j(\kappa))2^\gamma\leq j(\kappa)\quad\text{and hence}\quad{}M\vDash\abs{\power(D)^M}\leq j(\kappa).
\end{equation*}
\(2^\lambda\leq\abs{\power(D)^M}\) as follows: Consider the function \(f\colon\power(\lambda)\to\power(D)^M\) given by \(f(A)=j(A)\cap D\). Since \(j``\lambda\subseteq D\) we have that if \(f(A)=f(B)\) then \(j``A=j``B\) and so \(A=B\). Hence \(f\) is an injection and \(2^\lambda\leq\abs{j(\kappa)}\).\footnote{This method is similar to \cite[Lemma~3.3.2]{jech_ideals_1979}, but could be older. We are grateful for Goldberg's help in \cite{goldberg_mo_control_2023} for this result.}

On the other hand, \(j(\kappa)=\Set{[f]_\calU\mid f\colon X\to\kappa}\) and so \(j(\kappa)<(\kappa^\lambda)^+=(2^\lambda)^+\). Thus \(2^\lambda\leq j(\kappa)<(2^\lambda)^+\) as required.
\end{proof}

\subsection{A class of measurable cardinals}\label{s:applications;ss:class-of-measurables}

Assume \(\GCH\) and suppose that \(\kappa<\lambda\) are the two smallest measurable cardinals. By \cite{levy_measurable_1967}, if \(G\) is \(V\)-generic for some \(\bbP\), where \(\abs{\bbP}<\lambda\), then \(\lambda\) is still measurable in \(V[G]\). Hence, as in \cref{prop:iteration-length-kappa}, if we force with \(\bigast_{\alpha\in A}\Add(\alpha,\alpha^+)\), where \(A=\Set{\alpha<\kappa\mid\alpha\text{ is regular}}\), then there will be a maximal \(\kappa\)-independent family \(\calA\subseteq\power(\kappa)\) in the forcing extension. Furthermore, this forcing has cardinality \(\kappa\) and so \(\lambda\) will still be measurable, and \(\GCH\) will still hold. If we were to repeat this, say letting \(\bbP'=\bigast_{\alpha\in A'}\Add(\alpha,\alpha^+)\) in the forcing extension, where \(A'=\Set{\alpha<\lambda\mid\kappa<\alpha\land\alpha\text{ is regular}}\), then again \cref{prop:iteration-length-kappa} shows that in a new forcing extension by \(\bbP'\) there is a maximal \(\lambda\)-independent family \(\calA'\subseteq\power(\lambda)\). However, since \(\bbP'\) is \(\kappa^+\)-closed, no new subsets of \(\kappa\) nor sequences of length \(\kappa\) in \(\calA\) have been added, so \(\calA\) is still maximal \(\kappa\)-independent in the second forcing extension. One may reasonably expect that we can continue iterating this procedure to produce a (potentially class-size) forcing extension \(V[G]\) such that, whenever \(\kappa\) is measurable in \(V\), there is a maximal \(\kappa\)-independent family on \(\kappa\).

The na\"ive approach to this argument has us construct the Easton-support iteration \(\bigast_{\alpha\in A}\Add(\alpha,\alpha^+)\), where \(A\) is the class of all regular non-measurable cardinals. We would then hope to use \cref{prop:iteration-length-kappa} to show that if \(G \subseteq \bbP\) is \(V\)-generic and \(\calU\) is a normal measure on \(\kappa\) then we can construct a maximal \(\kappa\)-independent family on \(\kappa\) in \(V[G]\). While this may work, one must be careful of the condition that \(A\in\calU\) found in \cref{prop:iteration-length-kappa}. If \(\calU\) was such that \({\Set{\alpha<\kappa\mid\alpha\text{ is measurable}}\in\calU}\), then \(\kappa\notin j_\calU(A\cap\kappa)\) and so \(j_\calU(\bbP_\kappa)\) is not isomorphic to \(\bbP_\kappa\iter\Add(\kappa,\kappa^+)\iter\ddbbR\) as desired. Instead \(j_\calU \colon V \to M\) may be lifted to \(\hj\colon V[G \res \kappa]\to M[G\res\kappa\iter F]\) in \(V[G\res\kappa]\) (and then can be lifted to \(\tj\colon V[G] \to M[j(G)]\) in \(V[G]\) by the closure of \(\bbP/\bbP_\kappa\)). Therefore we must be sure to use \(\calU\) with \(A\cap\kappa\in\calU\) in our argument, which is to say \(o(\calU)=0\). Fortunately, such such measures always exist by the well-foundedness of \(\mitchless\).

Continuing along our lifting argument, if \(\calU \in V\) is a normal measure on \(\kappa\) and \(o(\calU)^V > 0\) then there is a normal measure \(\hat\calU \supseteq \calU\) in \(V[G]\). This allows us to show that if \(o(\kappa)^V > \alpha\) then \(o(\kappa)^{V[G]} \geq \alpha\). Though not all Mitchell ranks are preserved (we shall see that if \(o(\kappa)^V = 1\) then \(o(\kappa)^{V[G]} = 0\)), the reduction shall be `minimal': A closure point argument \`a la \cref{prop:hamkins-closure-gives-unlifting} gives us that if \(\calU \in V[G]\) is a normal measure in the forcing extension then \(\calU \cap V \in V\) is a normal measure in \(V\). Hence if \(o(\kappa)^V>0\) then \(o(\kappa)^V=1+o(\kappa)^{V[G]}\) exactly. That is, \(o(\kappa)^V=o(\kappa)^{V[G]}-1\) if \(o(\kappa)^V\) is positive and finite, and otherwise \(o(\kappa)^V=o(\kappa)^{V[G]}\). This operation warrants some ad-hoc notation. For \(\alpha\in\Ord\), let
\begin{equation*}
\premin{\alpha}\defeq\begin{cases}
0&\alpha=0\\
\alpha-1&0<\alpha<\omega\\
\alpha&\omega\leq\alpha.
\end{cases}
\end{equation*}
Our suggested interpretation of this operation is that, given some well-founded relation \(\tup{X,{\prec}}\), we may produce a new relation \(\tup{\premin{X},{\prec}}\) by setting \(\premin{X}\) to be those \(x \in X\) that are not minimal with respect to \({\prec}\). Then if \(\alpha\) is the height of \({\prec}\) on \(X\), \(\premin{\alpha}\) is the height of \({\prec}\) restricted to \(\premin{X}\).

The only other consideration is \(\GCH\). However, this is easy to force while preserving the Mitchell rank of all cardinals, such as with the Easton-support iteration \(\bigast_{o(\kappa)>0}\Add(\kappa^+,1)\).

\begin{thmalph}\label{thm:measurable-class-families}
Let \(V\) be a model of \(\ZFC+\GCH\). Then there is a class-length forcing iteration \(\bbP\) preserving \(\ZFC+\GCH\) such that, if \(G\subseteq\bbP\) is \(V\)-generic, then whenever \(\kappa\) is a measurable cardinal in \(V\) there is a maximal \(\kappa\)-independent family \(\calA\subseteq\power(\kappa)\) in \(V[G]\). Furthermore, whenever \(\kappa\) is a measurable cardinal in \(V\), \(o(\kappa)^V=1+o(\kappa)^{V[G]}\), and whenever \(\kappa\) is non-measurable in \(V\) it remains non-measurable in \(V[G]\).
\end{thmalph}

\begin{proof}
Let us first define our iteration system \(\tup{\bbP_\alpha,\ddbbQ_\alpha\mid\alpha\in\Ord}\). For all regular non-measurable cardinals \(\alpha\), let \(\ddbbQ_\alpha\) be a \(\bbP_\alpha\)-name for \(\Add(\alpha,\alpha^+)\) in the extension, and let \(\ddbbQ_\alpha=\Set{\1}\) otherwise. We iterate this with Easton support: At limit stage \(\alpha\), if \(\alpha\) is regular, let \(\bbP_\alpha\) be the direct limit of all \(\bbP_\beta\) for \(\beta<\alpha\), otherwise let \(\bbP_\alpha\) be the inverse limit. Let \(\bbP\) be the direct limit of all \(\bbP_\alpha\).

Note that for all regular \(\alpha\), \(\bbP=\bbP_\alpha\iter(\bbP/\bbP_\alpha)\), where \(\bbP_\alpha\) is \(\alpha^+\)-c.c.\ and \(\bbP/\bbP_\alpha\) is forced to be \(\alpha\)-closed. By a standard application of forcing techniques we have that \(\bbP\) is tame and thus will preserve \(\ZFC\).\footnote{\cite[Ch.~15]{jech_set_2003} provides a comprehensive overview of preservation of \(\ZFC\) using class products. \cite{friedman_fine_2000} has a deep treatment of class length forcing iterations.} For measurable \(\kappa\), we also have that \(\bbP=\bbP_\kappa\iter(\bbP/\bbP_\kappa)\) with \(\bbP/\bbP_\kappa\) forced to be \(\kappa^+\)-complete. Therefore, if \(\bbP_\kappa\) adds a maximal \(\kappa\)-independent family \(\calA\subseteq\power(\kappa)\) then, after forcing with \(\bbP/\bbP_\kappa\), \(\calA\) will still be maximal \(\kappa\)-independent. It therefore remains to show that \(\bbP_\kappa\) does indeed add such a family, in the manner of \cref{prop:iteration-length-kappa}. However, after taking care to pick a measure \(\calU \in V\) with \(o(\calU)^V=0\), we may apply \cref{prop:iteration-length-kappa} without modification. In this case, \(A = \Set{\lambda \in V \mid o(\lambda)^V > 0}\), so \(A \cap \kappa \in \calU\).

The rest of the proof will be spent showing that, for all \(\kappa\), \(o(\kappa)^{V[G]}=\premin{o(\kappa)^V}\). Let us begin by noting that for all \(\kappa\), \(\bbP/\bbP_{\kappa^+}\) is \(\kappa\)-closed, and \(\bbP_{\kappa^+}=\bbP_\omega\iter(\bbP_{\kappa^+}/\bbP_\omega)\), where \(\bbP_\omega=\Add(\omega,\omega_1)\) and \(\bbP_{\kappa^+}/\bbP_\omega\) is \(\sigma\)-closed. That is, \(\bbP_{\kappa^+}\) has a closure point at \(\omega\). Therefore, if \(\calU \in V[G\res\kappa^+]\) is a normal measure on \(\kappa\) then, by \cref{prop:hamkins-closure-gives-unlifting}, \(\calU \cap V \in V\) is a normal measure on \(\kappa\) in \(V\). Since \(\bbP/\bbP_{\kappa^+}\) is \(\kappa^+\)-closed, any normal measure on \(\kappa\) in \(V[G]\) must have already been present in \(V[G\res\kappa^+]=V[G\res\kappa]\).\footnote{Having found out that \(\kappa\) is measurable in \(V\), we conclude that \(\ddbbQ_\kappa=\Set{\1}^\bullet\).} In particular, if \(o(\kappa)^V = 0\) then \(o(\kappa)^{V[G]} = 0\). Having established this, the following claim will be helpful for our lifting arguments.

\begin{claim}\label{claim:kappa-minus-a-in-u}
If \(\calU \in V[G]\) is a normal measure on \(\kappa\), then
\begin{equation*}
\kappa \setminus A = \Set{\lambda < \kappa \mid o(\lambda)^V > 0} \in \calU.
\end{equation*}
\end{claim}

\begin{poc}
By prior calculations let us work in \(V[G\res\kappa]\) and let
\begin{equation*}
j = j_\calU \colon V[G\res\kappa] \to N=M[j(G\res\kappa)]
\end{equation*}
be the associated ultrapower embedding. Note that
\begin{equation*}
j(\bbP_\kappa)=\bigast_{\alpha\in j(A\cap\kappa)}\Add(\alpha,(\alpha^+)^N)^N.
\end{equation*}
By the \(\kappa^+\)-closure of \(N\) in \(V[G\res\kappa]\), if \(\kappa\in j(A)\) then \(\Add(\kappa,\kappa^+)^{V[G\res\kappa]}\) is an iterand of \(j(\bbP_\kappa)\) and we can extract from \(j(G)\) a \(V[G\res\kappa]\)-generic filter for \(\Add(\kappa,\kappa^+)^{V[G\res\kappa]}\). However, \(j\) is definable in \(V[G\res\kappa]\) and so certainly such an object cannot exist in \(V[G\res\kappa]\). Hence, \(\kappa\notin j(A)\) and so \(\kappa\setminus A\in \calU\).
\end{poc}

The rest of the proof shall be spent showing the exact Mitchell ranks of cardinals in \(V[G]\): For all \(\kappa\), \(o(\kappa)^{V[G]}=\premin{o(\kappa)^V}\). We shall do this by induction, so suppose that for all \(\lambda<\kappa\), \(o(\lambda)^{V[G]}=\premin{o(\lambda)^V}\). As we have shown that \(o(\kappa)^V=0\) implies that \(o(\kappa)^{V[G]}=0\), let us assume that \(o(\kappa)^V>0\).

\((o(\kappa)^{V[G]}\leq\premin{o(\kappa)^V})\). Suppose that \(o(\kappa)^{V[G]}>\premin{o(\kappa)^V}\), witnessed by normal measure \(\calU \in V[G]\) such that \(o(\calU)^{V[G]} = \premin{o(\kappa)^V}\). By \cref{claim:kappa-minus-a-in-u}, \(\kappa \setminus A \in \calU \cap V\), and hence \(o(\calU\cap V)^V > 0\) and \(o(\kappa)^V > 1\). In particular, for any \(\alpha\), if \(\premin{\alpha}=\premin{o(\kappa)^V}\) then \(\alpha=o(\kappa)^V\). Therefore,
\begin{align*}
\Set{\lambda < \kappa \mid o(\lambda)^{V[G]} = \premin{o(\kappa)^V}}&=\Set{\lambda < \kappa \mid \premin{o(\lambda)^V} = \premin{o(\kappa)^V}}\\
&=\Set{\lambda < \kappa \mid o(\lambda)^V = o(\kappa)^V}\\
&\in\calU \cap V,
\end{align*}
and so \(o(\calU \cap V)^V = o(\kappa)^V\), a contradiction.

\((o(\kappa)^{V[G]}\geq\premin{o(\kappa)^V})\). By \cref{prop:iteration-length-kappa}, if \(\calU \in V\) is a normal measure on \(\kappa\) such that \(A \cap \kappa \notin \calU\) (i.e.\ \(o(\calU)^V>0\)), there is \(\hat\calU \supseteq \calU\) a normal measure on \(\kappa\) in \(V[G\res\kappa]\). Furthermore, since \(\bbP/\bbP_\kappa\) is \(\kappa^+\)-closed, \(\hat\calU\) is still a normal measure on \(\kappa\) in \(V[G]\). Since \(o(\kappa)^{V[G]}\geq0\) by definition, let us assume that \(o(\kappa)^V > 1\) and prove that \(o(\kappa)^{V[G]} \geq \premin{o(\kappa)^V}\). If \(\calU \in V\) is such that \(o(\calU)^V > 0\) then
\begin{equation*}
\Set{\lambda < \kappa \mid o(\lambda)^V = o(\calU)^V}=\Set{\lambda< \kappa \mid o(\lambda)^{V[G]} = \premin{o(\calU)^V}}\in \hat\calU.
\end{equation*}
Hence, for all \(\alpha<o(\kappa)^V\), \(\premin{\alpha} < o(\kappa)^{V[G]}\), so \(\premin{o(\kappa)^V} \leq o(\kappa)^{V[G]}\) as required.
\end{proof}

Note that this result on the Mitchell rank may not be reversible. Let \(\calU,\calV\in V\) be any two normal measures on some \(\kappa\) with \(o(\kappa)^V=1\), and \(A\in\calU\setminus\calV\) a set of regular cardinals. Then forcing with the Easton-support iteration \(\bigast_{\alpha\in A}\Add(\alpha,\alpha^+)\) will produce a maximal \(\kappa\)-independent family \(\calA\subseteq\power(\kappa)\) thanks to \(\calU\), but \(\hat\calV\) will witness that \(\kappa\) is measurable in the forcing extension. However, there need not be an inner model witnessing \(o(\kappa)>1\). On the other hand, if there is a normal measure \(\calU\) on \(\kappa\) such that
\begin{equation*}
A=\Set*{\lambda<\kappa\mid(\exists\calA\subseteq\power(\lambda))\calA\text{ is maximal $\lambda$-independent}}\in\calU,
\end{equation*}
then it seems likely that \(o(\kappa)=2\) in the model \(L[\tup{\calI_\lambda\mid\lambda\in A},\calI_\kappa,\calU]\), where \(\calI_\lambda\) is the \(\lambda^+\)-saturated, \(\lambda\)-complete ideal on \(\lambda\) given by \cref{thm:family-gives-ideal} (see \cref{qn:mitchell-order-inner-model}).

\section{The future}\label{s:future}

Kunen's equiconsistency of measurable cardinals with maximal \(\sigma\)-independent families opens up an exciting correspondence between the consistency strength of large cardinals and corresponding collections of maximal \(\theta\)-independent families, either as a study in its own right or as an avenue to analyse other consistency strength relationships (such as determinacy or forcing axioms).

\begin{qn}
Can we extend the methods of \cref{thm:measurable-class-families} and \cref{prop:iteration-length-kappa} to produce a proper class of maximal \(\mu\)-independent families when \(\kappa\) is \(\mu\)-strongly compact but not \(\mu^+\)-strongly compact? If \(\kappa\) is strongly compact then can we obtain a model in which there is a proper class of maximal \(\kappa\)-independent families?
\end{qn}

\begin{qn}
Can the technique of \cref{thm:kunens-theorem-applicable} be extended to general elementary embeddings, rather than ultrapower embeddings? If so, can we soften the requirements of \cref{thm:main-sc} to only requiring, say, a \(\theta\)-strong cardinal?
\end{qn}

\begin{qn}
What is the consistency strength of a proper class of maximal \(\sigma\)-independent families? Of a proper class of \(\theta\) such that there is a maximal \(\theta\)-independent family? Of a proper class of \(\theta\) such that there is a maximal \(\theta\)-independent family on \(\theta\)?
\end{qn}

Our hope would be that one could investigate these questions without needing to develop an inner model theory for strongly compact cardinals, which would pose a non-trivial obstacle. Similarly, we would like to know more about \(\theta\)-strongly compact cardinals and where they fall in the large cardinal hierarchy.

\begin{qn}
Is the consistency strength of a \(\theta\)-strongly compact cardinal strictly lower than that of a strongly compact cardinal?
\end{qn}

\begin{qn}
Is the least \(\theta\)-strongly compact cardinal a strong limit? In \cite[Proposition~6.2]{gitik_sigma_2020} Gitik constructs (from \(\kappa<\lambda\) supercompact plus \(\GCH\)) a cofinality-preserving forcing extensions such that \(2^\kappa=\lambda\) and \(\lambda\) is \(\kappa\)-strongly compact. Is \(\lambda\) (consistently) the least \(\kappa\)-strongly compact cardinal in this model?
\end{qn}

\begin{qn}\label{qn:mitchell-order-inner-model}
When can we recover Mitchel order from models with many \(\kappa\) such that there is a maximal \(\kappa\)-independent family on \(\kappa\)? For example, is it sufficient to have \(\calA\in\Ult(V,\calU)\) to produce an inner model in which \(o(\kappa)\geq2\)?
\end{qn}

\section{Acknowledgements}\label{s:acknowledgements}

The author would like to thank Gabe Goldberg for his help with the proof of \(2^\lambda\leq\abs{\power(D)^M}\) in the proof of \cref{thm:main-sc}. The author would like to thank Andrew Brooke-Taylor and Asaf Karagila for their comments on early versions of this paper.

\providecommand{\bysame}{\leavevmode\hbox to3em{\hrulefill}\thinspace}
\providecommand{\MR}{\relax\ifhmode\unskip\space\fi MR }
\providecommand{\MRhref}[2]{%
  \href{http://www.ams.org/mathscinet-getitem?mr=#1}{#2}
}
\providecommand{\href}[2]{#2}

\end{document}